\documentclass[preprint,11pt]{elsarticle}

\usepackage{amsmath}
\AtBeginDocument{%
   \setlength\abovedisplayskip{3pt}
   \setlength\belowdisplayskip{3pt}}

 \usepackage[margin=1.35in,footskip=0.25in]{geometry}
\usepackage{amssymb,color}
\usepackage{subfig}
\usepackage{amsthm}
\usepackage{mathtools}
\usepackage{graphicx}
\newcommand{\e}{\mathrm{e}}

\newcommand{\eps}{\varepsilon}
\newcommand{\norm}[2]{\|{#1}\|_{#2}}
\newcommand{\tnorm}[1]{\left|\!\!\;\left|\!\!\;\left| {#1}
                       \right|\!\!\;\right|\!\!\;\right|}

\newcommand{\U}{\mathcal{U}}
\newcommand{\PS}{\mathcal{P}}
\newcommand{\scp}[1]{\left\langle #1 \right\rangle}

\newtheorem{theorem}{Theorem}[section]
\newtheorem{lemma}[theorem]{Lemma}

\newtheorem{remark}[theorem]{Remark}
\newtheorem*{Notation*}{Notation}

\journal{}

\begin{document}

\begin{frontmatter}

%% Title, authors and addresses
%% \title{Title\tnoteref{label1}}
%% \tnotetext[label1]{}
%% \author{Name\corref{cor1}\fnref{label2}}
%% \ead{email address}
%% \ead[url]{home page}
%% \fntext[label2]{}
%% \cortext[cor1]{}
%% \address{Address\fnref{label3}}
%% \fntext[label3]{}

% \title{A convection-diffusion problem with large shift on Dur\'{a}n mesh}
\title{A convection-diffusion problem with a large shift on Dur\'{a}n meshes}
%\tnotetext[t1]{This paper has been supported by the Ministry of Education, Science and Technological Development of the Republic of Serbia, project no.
 %451-03-9/2021-14/200134 and project no. 451-03-68/2020-14/200156: "Innovative scientific and artistic research
	%from the FTS activity domain".}
\author[MBr]{Mirjana Brdar\corref{cor1}}
\ead{mirjana.brdar@uns.ac.rs}
\author[SF]{Sebastian Franz}
\ead{ sebastian.franz@tu-dresden.de}
\author[HGR]{Hans-G\"{o}rg Roos}
\ead{hans-goerg.roos@tu-dresden.de}
\address[MBr]{Faculty of Technology Novi Sad, University of Novi Sad,
Bulevar cara Lazara 1, 21000 Novi Sad, Serbia}
\address[SF]{Institute of Scientific Computing, Technische Universit\"at Dresden, Germany}
\address[HGR]{Institute of Numerical Mathematics, Technische Universit\"at Dresden, Germany}
\cortext[cor1]{Corresponding author}

\begin{abstract}
A convection-diffusion problem with a large shift in space is considered. Numerical analysis of high order finite element methods on layer-adapted Dur\'{a}n type meshes, as well as on coarser Dur\'{a}n type meshes in places where weak layers appear, is provided. The theoretical results are confirmed by numerical experiments.
\end{abstract}

\begin{keyword}
	%% keywords here, in the form: keyword \sep keyword
	spatial large shift \sep singularly perturbed \sep  Galerkin method \sep Dur\'{a}n mesh
	%% PACS codes here, in the form: \PACS code \sep code
	
	%% MSC codes here, in the form: \MSC code \sep code
	%% or \MSC[2008] code \sep code (2000 is the default)
\end{keyword}

\end{frontmatter}

%----------------------------------------------------------------------
\section{Introduction}
%----------------------------------------------------------------------

\label{sec:intro}
Singularly perturbed problems with delay are differential equations that allow
past actions to be included in mathematical models, bringing the model closer to real-world
phenomena. Their solution depends not only on on the solution at a current stage, but also on the
solution at some past stages. This type of problem occurs in the control theory and biosciences
\cite{Glizer, LongtinMilton, Wazewska}, and in the study of chemostat models, circadian rhythms,
epidemiology, the respiratory system, tumour growth and neural networks. Although singularly
perturbed problems have been studied intensively in recent years, there are not many papers on
these problems with a on these problems with large displacements.

Here we consider the following convection-diffusion problem with a large shift: Find $u$ such that
  \begin{subequations}\label{prob}
     \begin{align}
       - \varepsilon u''(x) -b(x)u'(x)+ c(x) u(x)+d(x)u(x-1)& =f(x), \quad x\in \Omega=(0,2), \label{1}\\
      u(2)&=0,\label{2}\\
      u(x)&=\Phi(x), \quad x\in (-1,0],\label{3}
          \end{align}
  \end{subequations}
  where $b\geq\beta>0,$ $d\geq 0$ and   $0<\varepsilon\ll 1$ is a small perturbation parameter.
  We assume $\displaystyle c-\frac{b'}{2}-\frac{\|d\|_{L^{\infty}(1,2)}}{2}\geq\gamma>0,$ as well as $\Phi(0) = 0$, which is not a restriction for the function $\Phi,$ since this condition can always be guaranteed by a simple transformation. Thus, it holds $u\in H_0^1(\Omega).$

 For convection dominated singularly perturbed problems with an additional shift you can find only a few paper, see \cite{RS, SR1, SR2}, where authors consider a negative coefficient $d$, which supports a maximum principle. Then they use the finite-difference method on layer-adapted meshes. Here we consider finite element methods of arbitrary
order for positive coefficients $d$. The solution decomposition for this kind of problem we proved in \cite{BFLR}, as well as the error analysis on a standard S-type mesh. Our problem has an exponential layer near $x=0$ and one weak layer near $x=1$ whose appearance is caused by the delay.

Numerical methods for singularly perturbed problems typically use some known
behavior of the exact solution to design a priori adapted meshes in order to approximate
boundary layers well. Probably the most well-known approximations of this kind are those based on the so-called Shishkin meshes (see \cite{RST} and references therein).
Dur\'{a}n and Lombardi in \cite{Duran} introduced a new graded mesh in order to solve a
convection-diffusion model problem by standard bilinear finite elements. They
proved an (almost) optimal error estimate. Up to the logarithmic factor, those estimates are valid
uniformly in the perturbation parameter. Hence, the graded meshes could be
an excellent replacement for the well-known piecewise uniform Shishkin mesh
which has transition point(s) dividing a domain on coarse and fine mesh regions.
Indeed, according to some numerical experiments, the graded mesh procedure
seems to be more robust in the sense that the numerical results are not strongly
affected by variations of parameters which define the mesh. This is the reason why we used the Dur\'{a}n mesh in the paper. Since the problem has a weak layer, we constructed a coarser Dur\'{a}n mesh based on \cite{BFLR, Roos}.

The outline of the paper is as follows. In Section 2, we give the solution decomposition and introduce the recursively defined graded mesh. The following section contains a description of the finite element method and error analysis for the stationary singularly perturbed shift problem and the main convergence result.
In Section 4 we present a coarser mesh of Dur\'{a}n-type and numerical analysis on this mesh for problem (\ref{prob}). Finally, Section 5 provides some numerical results supporting our  theoretical analysis.

  \vspace{2mm}

  \begin{Notation*}
    We use the standard notation of Sobolev spaces, where, for a set $D$,  $\|\cdot\|_{L^2(D)}$ is the $L^2-$norm and $\scp{\cdot,\cdot}_D$
    is the standard scalar product in $L^2(D).$ Also, we write $A\lesssim B$ if there exists a generic positive constant $C$ independent of the perturbation parameter $\varepsilon$ and the mesh,
    such that $A\leq C B$.
  \end{Notation*}

%%%%%%%%%%%%%%%%%%%%%%%%%%%%%%%%%%%%%%%%%%%%%%%%%%%%%%%%%%%%%%%%%%%%%%%%%%%%%%%%%%%%%%%%%%%%%%%%%%%%%%%%%%%%%%

  \section{Solution decomposition and mesh}
 The solution decomposition of this problem proved in \cite{BFLR} we give in the following theorem.
  \begin{theorem}\label{th1} Let $k\geq 0$ be a given integer and the data of (\ref{prob}) smooth enough. Then it holds
     \[
      u=S+E+W,
    \]
   where for any $l\in\{0,1,\dots,k\}$ it holds
  \begin{align*}
      \norm{S^{(\ell)}}{L^2(0,1)}+\norm{S^{(\ell)}}{L^2(1,2)} & \lesssim 1,&
      |E^{(\ell)}(x)|&\lesssim \eps^{-\ell}\e^{-\beta \frac{x}{\eps}},\quad x\in[0,2],\\
      &&|W^{(\ell)}(x)|&\lesssim \begin{cases}
                                   0,&x\in(0,1),\\
                                   \eps^{1-\ell}\e^{-\beta \frac{(x-1)}{\eps}},& x\in(1,2).
                                 \end{cases}
    \end{align*}
  \end{theorem}
  Thus, $E$ is the layer function corresponding to the left boundary, while
  $W$ is an interior layer function, and $S$ represents the smooth part.

  Knowing how the layers are structured allows us to create a layer-adapted mesh that resolves the layers. We will modify the recursively defined graded mesh of the Dur\'{a}n type to our problem. Its advantages are the simple construction and generation of mesh points (without transition point(s)) and a certain robustness property. More precisely, if we are approximating a singularly perturbed problem with a mesh that has been adapted a priori,  we can expect that a mesh designed for one value of the small parameter will perform well for larger values of the small parameter as well. In this respect, the recursively graded meshes have better behavior in numerical experiments.

  We first define the points on the interval $[0, 1]$ and then the rest of the domain to construct this mesh.
  Let $H\in(0,1)$ be arbitrary. Define the number M
%   from the conditions
%   \begin{subequations}\label{M}
%     \begin{equation}
%       H\eps(1+H)^{M-2}<1\leq      H\eps(1+H)^{M-1}
%     \end{equation}
%     which are equivalent to
%     \begin{equation}
%       M=\left\lceil1-\frac{\ln(H\eps)}{\ln(1+H)}\right\rceil.
%     \end{equation}
%   \end{subequations}

  by
  \begin{equation}\label{M}
    M=\left\lceil\left\lceil\frac{1}{H}\right\rceil - \frac{\ln\left(H\eps\left\lceil\frac{1}{H}\right\rceil\right)}{\ln(1+H)}\right\rceil.
  \end{equation}
%   On the interval $[0,1]$ we define mesh points recursively in the following way:
  We define mesh points recursively in the following way:
  \begin{equation}\label{Dmesh1}
    x_i=\begin{cases}
          0, & i=0,\\
         i H\eps,& i=1\leq i \leq\left\lceil\frac{1}{H}\right\rceil ,\\
          (1+H)x_{i-1}, & \left\lceil\frac{1}{H}\right\rceil < i\leq M-1,\\
          1,& i=M,\\
          1+x_{i-M},& M\leq i\leq 2M.\\
        \end{cases}
  \end{equation}
  If the interval $(x_{M-1}, 1)$ is too small in relation to $(x_{M-2}, x_{M-1})$, we simply omit the mesh point $x_{M-1}$. The total number of mesh subintervals is $N=2M.$ It depends on the parameters $H$ and conditions \eqref{M}, see \cite{Duran}. Furthermore, the inequality
  \begin{equation}\label{H}
    H\lesssim N^{-1}\ln(1/\eps)
  \end{equation}
  holds.

   For mesh step sizes $h_i=x_i-x_{i-1}$, $1\leq i\leq 2M$, is valid $CH\eps\leq h_i\leq H,$ where $C$ is a constant independent of $\eps$ and the number of mesh points. It also applies
  \begin{equation}%\label{eq:step}
    \begin{array}{lll}
      h_i=H\eps, & i\in\{1,\dots,\left\lceil\frac{1}{H}\right\rceil\}\cup\{M+1,\dots,M+1+\left\lceil\frac{1}{H}\right\rceil\},\\[0.5ex]
      h_i\leq Hx,&  i\in\{\left\lceil\frac{1}{H}\right\rceil+1,\dots,M\}\cup\{M+2+\left\lceil\frac{1}{H}\right\rceil,\dots,2M\},\\[0.5ex]
          \end{array}
  \end{equation}
where $x\in[x_{i-1},x_i].$

\begin{remark}
For an arbitrarily chosen parameter H in the Dur\'{a}n mesh, $2M$ steps of the mesh are obtained, which is generally not comparable to the number of steps in other papers that deal with the same or similar issues. An approach that eliminates this shortcoming is given in the paper \cite{RBT}.
\end{remark}

\section{Finite element method and error estimates}

The bilinear and linear form for problem (\ref{prob}) are given by
  \begin{align}
    B(u,v)&:=\eps\scp{u',v'}_{\Omega}+\scp{cu-bu',v}_{\Omega}+\scp{du(\cdot -1),v}_{(1,2)}\notag\\
          &= \scp{f,v}_{\Omega}-\scp{d\phi(\cdot -1),v}_{(0,1)}=:F(v)
  \end{align}
  for $u,v\in \U$.

  We use $\U_N:=\{v\in H_0^1(\Omega):v|_\tau\in\PS_k(\tau)\}$ for discrete space,
  where $\PS_k(\tau)$ is the space of polynomials of degree $k$ at most on a cell $\tau$ of the mesh.
  Let $I$ be the standard Lagrange-interpolation operator into $\U_N$ using equidistant points or any other suitable
  distribution of points.

  The finite element method is given by: Find $u_N\in \U_N$ such that for all $v\in \U_N$ it holds
  \begin{equation}\label{eq:method}
    B(u_N,v)=F(v).
  \end{equation}

  The bilinear form is coercive in the energy norm $\tnorm{u}^2:=\eps\norm{u'}{L^2(\Omega)}^2 + \gamma \norm{u}{L^2(\Omega)}^2$ and satisfies the Galerkin orthogonality $B(u-u_N,v)=0$ for all $v\in \U_N.$

  \begin{lemma}\label{le1}
    For the standard piecewise Lagrange interpolation operator on the graded mesh \eqref{Dmesh1} we have
    \begin{align}
        \norm{u-Iu}{L^2(\Omega)} & \lesssim H^{k+1},\label{D1}\\
        \norm{(u-Iu)'}{L^2(\Omega)} & \lesssim {\eps}^{-1/2}H^k.\label{D2}
    \end{align}
  \end{lemma}
  \begin{proof}
    In the proof we use the norm definitions, the solution decomposition and standard interpolation error estimates, given on any cell $\tau_i$ with width $h_i$ and $1\leq s\leq k+1$ and $1\leq t\leq k$ by
    \begin{subequations}\label{inter}
      \begin{align}
        \|v-Iv\|_{L^2(\tau_i)}&\lesssim h_i^s\|v^{(s)}\|_{L^2(\tau_i)},\label{int1}\\
            \|(v-Iv)'\|_{L^2(\tau_i)}&\lesssim h_i^t\|v^{(t+1)}\|_{L^2(\tau_i)},\label{int2}
      \end{align}
    \end{subequations}
    for $v$ smooth enough. We obtain
    \begin{align*}
      &\|S-IS\|_{L^2(0,1)}^2
        =  \sum\limits_{i=1}^M\|S-IS\|_{L^2(I_i)}^2\\
        &\lesssim \sum\limits_{i=1}^{\left\lceil\frac{1}{H}\right\rceil}h_i^{2(k+1)}\|S^{(k+1)}\|_{L^2(I_i)}^2+ \sum\limits_{i=\left\lceil\frac{1}{H}\right\rceil+1}^M h_i^{2(k+1)}\|S^{(k+1)}\|_{L^2(I_i)}^2\\
        &\lesssim  \sum\limits_{i=1}^{\left\lceil\frac{1}{H}\right\rceil}(H\eps)^{2(k+1)}\|S^{(k+1)}\|_{L^2(I_i)}^2+ \sum\limits_{i=\left\lceil\frac{1}{H}\right\rceil+1}^M H^{2(k+1)}\|x^{k+1}S^{(k+1)}\|_{L^2(I_i)}^2
                  \lesssim   H^{2(k+1)}
    \end{align*}
    where $I_i=(x_{i-1},x_i).$ The same estimate we get on $[1,2].$ In the same way we get
    \[\|(S-IS)'\|_{L^2(\Omega)}^2  \lesssim  H^{2k}.\]
    For the boundary layer part we obtain
    \begin{align*}
      \|&E-IE\|_{L^2(0,1)}^2
        \lesssim
        \sum\limits_{i=1}^{\left\lceil\frac{1}{H}\right\rceil} (H\eps)^{2(k+1)}\|E^{(k+1)}\|_{L^2(I_i)}^2+ \sum\limits_{i=\left\lceil\frac{1}{H}\right\rceil+1}^M H^{2(k+1)}\|x^{k+1}E^{(k+1)}\|_{L^2(I_i)}^2 \\
         &\lesssim  (H\eps)^{2(k+1)}\int\limits_0^{\eps}\eps^{-2(k+1)}e^{-\frac{2\beta x}{\eps}}dx+ H^{2(k+1)}\int\limits_{\eps}^{1}x^{2(k+1)}\eps^{-2(k+1)}e^{-\frac{2\beta x}{\eps}}dx
       \lesssim \eps H^{2(k+1)}.
    \end{align*}
    and the same estimate on $[1,2]$. Also, we get

    \begin{align*}
      \|(E-IE)'\|_{L^2(0,\eps)}^2
        \lesssim &
         \sum\limits_{i=1}^{\left\lceil\frac{1}{H}\right\rceil} (H\eps)^{2k}\|E^{(k+1)}\|_{L^2(I_i)}^2
        \lesssim  \eps^{-1}H^{2k},\\
%     \end{align*}
%
%     \begin{align*}
      \|(E-IE)'\|_{L^2(\eps,1)}^2
        \lesssim &
         \sum\limits_{i=\left\lceil\frac{1}{H}\right\rceil+1}^M H^{2k}\|x^{k+1}E^{(k+1)}\|_{L^2(I_i)}^2
    \lesssim \eps H^{2k}.
    \end{align*}
    For interior layer function we have following estimates
    \begin{align*}
      \|&W-IW\|_{L^2(1,2)}^2\\
        &\lesssim
        \sum\limits_{i=M}^{M+\left\lceil\frac{1}{H}\right\rceil} (H\eps)^{2(k+1)}\|W^{(k+1)}\|_{L^2(I_i)}^2
                + \sum\limits_{i=M+\left\lceil\frac{1}{H}\right\rceil+1}^{2M} H^{2(k+1)}\|x^{k+1}W^{(k+1)}\|_{L^2(I_i)}^2\\
        &\lesssim (H\eps)^{2(k+1)}\int\limits_1^{1+\eps}(\eps^{1-(k+1)})^2e^{-\frac{2\beta (x-1)}{\eps}}dx
               + H^{2(k+1)}\int\limits_{1+\eps}^{2}x^{2(k+1)}(\eps^{1-(k+1)})^2e^{-\frac{2\beta (x-1)}{\eps}}dx\\
        &\lesssim  \eps^3 H^{2(k+1)},\\
%     \end{align*}
%
%     \begin{align*}
          \|&(W-IW)'\|_{L^2(1,2)}^2
                \lesssim  (H\eps)^{2k}\int\limits_1^{1+\eps}\eps^{-2k}e^{-\frac{2\beta (x-1)}{\eps}}dx+H^{2k}\int\limits_{1+\eps}^2\eps^{-2k}x^{2k}e^{-\frac{2\beta (x-1)}{\eps}}dx
                \lesssim \eps H^{2k}.
    \end{align*}

    Using Theorem \ref{th1} and above estimates the statement of the lemma follows.
  \end{proof}

\begin{theorem}\label{th2}
  For the solution $u$ of problem (\ref{prob}) and the numerical solution $u_N$ of (\ref{eq:method}) on a Dur\'{a}n mesh (\ref{Dmesh1}) it holds
  \[\tnorm{u-u_N}\lesssim H^k.\]
\end{theorem}
\begin{proof}
  Based on the triangle inequality $\tnorm{u-u_N}\leq\tnorm{u-Iu}+\tnorm{Iu-u_N}$, and using the evaluation of Lemma \ref{le1} for the first term, it is sufficient to evaluate the second term. Let $\eta:=u-Iu$ and $\chi:=Iu-u_N\in \U_N.$ Coercivity and Galerkin orthogonality with parts of the proof of Lemma \ref{le1} yield
  \begin{align*}
    \tnorm{\chi}&\leq B(\chi,\chi)=B(\eta,\chi)=\eps\scp{\eta',\chi'}_{\Omega}+\scp{c\eta-b\eta',\chi}_{\Omega}+\scp{d\eta(\cdot -1),\chi}_{(1,2)}\\
    &\lesssim H^k\tnorm{\chi}+\scp{b(E-IE),\chi'}_{\Omega}\lesssim H^k.\qedhere
  \end{align*}
\end{proof}

\section{On a coarser mesh}

Following the idea from \cite{Roos} for the weak layer, we use a mesh that is sparser than the one defined in (\ref{Dmesh1}). This is called a coarser mesh. Numerical results show that for small $k$ and reasonably small $\eps$ it is possible to use coarser meshes. For higher polynomial degrees, the weak layer should be resolved by a classical layer-adapted mesh.

Here, we construct the mesh in the following way.
Let $M$ be defined as in \eqref{M} and analogously
\begin{equation}\label{M2}
  M_2=\left\lceil\left\lceil\frac{1}{H}\right\rceil - \frac{\ln\left(H\eps^{\frac{k-1}{k}}\left\lceil\frac{1}{H}\right\rceil\right)}{\ln(1+H)}\right\rceil.
\end{equation}
Then the mesh nodes are given by

\begin{equation}\label{Dmesh2}
    x_i=\begin{cases}
          0, & i=0,\\
         i H\eps,& 1\leq i \leq\left\lceil\frac{1}{H}\right\rceil ,\\
          (1+H)x_{i-1}, & \left\lceil\frac{1}{H}\right\rceil < i\leq M-1,\\
          1,& i=M,\\
          1+(i-M)H\eps^{\frac{k-1}{k}},& M+1\leq i \leq M+\left\lceil\frac{1}{H}\right\rceil ,\\
           1+(1+H)(x_{i-1}-1), & M+\left\lceil\frac{1}{H}\right\rceil < i\leq  M+M_2-1,\\
          2,& i= M+M_2.
        \end{cases}
  \end{equation}

The mesh step sizes $h_i=x_i-x_{i-1}$ satisfy
\begin{equation}%\label{eq:step}
    \begin{array}{lll}
      h_i=H\eps, & i\in\{1,\dots,\left\lceil\frac{1}{H}\right\rceil\}\\[0.5ex]
      h_i\leq Hx,&  i\in\{\left\lceil\frac{1}{H}\right\rceil+1,\dots,M\}\\[0.5ex]
      h_i=H\eps^{\frac{k-1}{k}},& i\in\{M+1,\dots,M+\left\lceil\frac{1}{H}\right\rceil\},\\[0.5ex]
      h_i\leq Hx,&  i\in\{M+\left\lceil\frac{1}{H}\right\rceil+1,\dots, M+M_2\}.
          \end{array}
  \end{equation}
where $x\in[x_{i-1},x_i].$

\begin{lemma}\label{le2}Let us assume $e^{-\eps^{-1/k}}\leq H^{k-1}.$ Then for the standard piecewise Lagrange interpolation operator on the graded mesh \eqref{Dmesh2} we have
          \begin{align}
        \norm{u-Iu}{L^2(\Omega)}
          & \lesssim H^{k+\frac{1}{2}},\label{D3}\\[0.5ex]
        \tnorm{u-Iu}
          & \lesssim H^k.\label{D4}
             \end{align}
      \end{lemma}
  \begin{proof}

  Similarly to the proof of the previous lemma, using solution decomposition and estimates \eqref{int1} and \eqref{int2}, we obtain same estimates for $\|S-IS\|_{L^2(0,2)},$ $\|E-IE\|_{L^2(0,2)}.$ Using the assumption $e^{-\eps^{-1/k}}\leq H^{k-1}$ we get
  \begin{align*}
      \|(E-IE)'\|_{L^2(1,1+\eps^{\frac{k-1}{k}})}
        \lesssim  \eps^{-\frac{1}{2}} H^{k}.
  \end{align*}
  For the estimation of $W$ we follow the idea given in \cite{Reibiger}. From
  \begin{align*}
      \|W-IW\|_{L^2(1+\eps^{\frac{k-1}{k}},2)}
        \lesssim &
        H\|xW'\|_{L^2(1+\eps^{\frac{k-1}{k}},2)} \lesssim \eps^{\frac{1}{2}} H^{k},
  \end{align*}
  where we use assumption $e^{-\eps^{-1/k}}\leq H^{k-1},$ and from
  \begin{align*}
      \|W-IW\|_{L^2(1+\eps^{\frac{k-1}{k}},2)}
        \lesssim &
        H^2\|x^2W''\|_{L^2(1+\eps^{\frac{k-1}{k}},2)} \lesssim \eps^{-\frac{1}{2}} H^{k+1},
  \end{align*}
  we obtain
  \begin{align*}
      \|W-IW\|_{L^2(1+\eps^{\frac{k-1}{k}},2)}
        \lesssim  H^{k+\frac{1}{2}}.
  \end{align*}
  For the derivative we obtain
  \begin{align*}
      \|(W-IW)'\|_{L^2(1,1+\eps^{\frac{k-1}{k}})}
        \lesssim &
        h_i^k\|W^{(k+1)}\|_{L^2(1,1+\eps^{\frac{k-1}{k}})} \lesssim \eps^{-\frac{1}{2}} H^{k},
  \end{align*}
  and
  \begin{align*}
      \|(W-IW)'\|_{L^2(1+\eps^{\frac{k-1}{k}},2)}
        \lesssim &
        h_i\|W''\|_{L^2(1+\eps^{\frac{k-1}{k}},2)} \lesssim \eps^{-\frac{1}{2}} H^{k}.
  \end{align*}
  Collecting all these estimates, we get the statement of lemma.
\end{proof}

\begin{theorem}\label{th3}
  For the solution $u$ of problem (\ref{prob}) and the numerical solution $u_N$ of (\ref{eq:method}) on a coarser Dur\'{a}n mesh it holds
  \[\tnorm{u-u_N}\lesssim H^k.\]
\end{theorem}

\begin{proof}
  Similar to the proof of the Theorem \ref{th2} using the result of the Lemma \ref{le2}.
\end{proof}

  \section{Numerical results}\label{sec:numerics}
As an example, let us look at the following problem taken from \cite{BFLR}
    \begin{align*}
      -\eps u''(x)-(2+x)u'(x)+(3+x)u(x)-d(x)u(x-1)&=3,\,x\in(0,2),\\
      u(2)&=0,\\
      u(x)&=x^2,\,x\in(-1,0],
    \end{align*}
    where
    \[
     d(x) = \begin{cases}
              1-x,& x<1,\\
              2+\sin(4\pi x),&x\geq 1.
            \end{cases}
    \]
   In this case the exact solution is not known. Our numerical simulations are performed using the finite-element framework $\mathbb{SOFE}$
    developed by L. Ludwig, see\\ \texttt{github.com/SOFE-Developers/SOFE}. We compare the
    results both on a Dur\'{a}n mesh and on a coarse Dur\'{a}n mesh. Since the exact solution to
    this example is unknown, we use a reference solution as a proxy. This is
    computed on a standard or coarsened Dur\'{a}n mesh with $k=4$ and $H=0.05$.
    Table~\ref{tab:Duran}
    \begin{table}[htbp]
      \caption{Errors for various polynomial degrees on standard Dur\'{a}n meshes and $\eps=10^{-6}$\label{tab:Duran}}
      \begin{center}
        \begin{tabular}{lr|llllll}
              &      & \multicolumn{2}{c}{$k=1$}           & \multicolumn{2}{c}{$k=2$}           & \multicolumn{2}{c}{$k=3$}\\
          $H$ & $\tilde M$ & \multicolumn{2}{c}{$\tnorm{u-u_N}$} & \multicolumn{2}{c}{$\tnorm{u-u_N}$} & \multicolumn{2}{c}{$\tnorm{u-u_N}$}\\
          \hline
          0.90 &   46 & 3.14e-01  & 1.20 &  6.83e-02 &  2.44 & 1.04e-02 &  3.67\\
          0.80 &   50 & 2.85e-01  & 1.05 &  5.57e-02 &  2.09 & 7.63e-03 &  3.39\\
          0.70 &   56 & 2.52e-01  & 0.99 &  4.40e-02 &  2.08 & 5.19e-03 &  2.27\\
          0.60 &   64 & 2.21e-01  & 1.19 &  3.33e-02 &  2.39 & 3.84e-03 &  3.64\\
          0.50 &   74 & 1.86e-01  & 1.24 &  2.36e-02 &  2.49 & 2.26e-03 &  3.37\\
          0.40 &   88 & 1.50e-01  & 1.17 &  1.53e-02 &  2.31 & 1.26e-03 &  3.40\\
          0.30 &  112 & 1.13e-01  & 1.08 &  8.77e-03 &  2.17 & 5.56e-04 &  3.64\\
          0.20 &  162 & 7.59e-02  & 1.06 &  3.94e-03 &  2.12 & 1.45e-04 &  3.23\\
          0.10 &  310 & 3.81e-02  &      &  9.96e-04 &       & 1.78e-05 &
        \end{tabular}
      \end{center}
    \end{table}
    shows the results for different polynomial degrees $k$ on the standard Dur\'{a}n mesh for $\eps=10^{-6}$. Given
    are, for different values of $H$, the corresponding number of cells $\tilde M=2M$, the error measured
    in the energy norm and the numerical rate of convergence. Obviously we have a convergence
    of order $k$. Table~\ref{tab:coarsenedDuran}
    \begin{table}[htbp]
      \caption{Errors for various polynomial degrees on coarse Dur\'{a}n meshes and $\eps=10^{-6}$\label{tab:coarsenedDuran}}
      \begin{center}
        \begin{tabular}{l|rllrllrll}
              & \multicolumn{3}{c}{$k=1$}
              & \multicolumn{3}{c}{$k=2$}
              & \multicolumn{3}{c}{$k=3$}\\
          $H$ & $\tilde M$ & \multicolumn{2}{c}{$\tnorm{u-u_N}$}
              & $\tilde M$ & \multicolumn{2}{c}{$\tnorm{u-u_N}$}
              & $\tilde M$ & \multicolumn{2}{c}{$\tnorm{u-u_N}$}\\
          \hline
          0.90 &  25 & 3.14e-01 & 1.31 &  35 & 6.85e-02 &  2.49 &   39 & 1.02e-02 &  3.43\\
          0.80 &  27 & 2.84e-01 & 1.11 &  38 & 5.58e-02 &  1.93 &   42 & 7.91e-03 &  3.64\\
          0.70 &  30 & 2.53e-01 & 1.11 &  43 & 4.40e-02 &  2.19 &   47 & 5.25e-03 &  2.06\\
          0.60 &  34 & 2.20e-01 & 1.23 &  49 & 3.31e-02 &  2.23 &   54 & 3.94e-03 &  4.89\\
          0.50 &  39 & 1.86e-01 & 1.07 &  57 & 2.36e-02 &  2.64 &   62 & 2.01e-03 &  3.31\\
          0.40 &  47 & 1.52e-01 & 1.20 &  67 & 1.54e-02 &  2.26 &   74 & 1.12e-03 &  3.42\\
          0.30 &  60 & 1.13e-01 & 1.03 &  86 & 8.75e-03 &  2.18 &   95 & 4.76e-04 &  3.41\\
          0.20 &  86 & 7.81e-02 & 1.10 & 124 & 3.94e-03 &  2.11 &  137 & 1.37e-04 &  3.16\\
          0.10 & 165 & 3.81e-02 &      & 238 & 9.96e-04 &       &  262 & 1.76e-05 &
        \end{tabular}
      \end{center}
    \end{table}
    shows the results on the coarse version of the Dur\'{a}n mesh with $\tilde M=M+M_2$ cells and $\eps=10^{-6}$.
    Again we observe a convergence of order $k$. Note that in the case of $k=1$ we essentially have an equidistant mesh
    in $[1,2]$ with $M_2=\left\lceil\frac{1}{H}\right\rceil$ cells.

    Comparing the results of the two tables, there is little difference in the calculated errors, but a large reduction
    in the number of cells. Thus, using the coarsened mesh gives equally good results with less computational cost.

  \vspace{1em}
  {\bf Acknowledgment}. The first author is supported by the Ministry of Science, Technological Development and Innovation of the Republic of Serbia under grant no. 451-03-47/2023-01/200134.


\begin{thebibliography}{99}
	
%\bibitem{Adams} R.A. Adams. Sobolev spaces. Academic Press, New York, 1975.

%\bibitem{AMM19} J.H. Adler, S. MacLachlan, and N. Madden. First-order system least squares finiteelements for singularly perturbed reaction-diffusion equations. In I. Lirkov and S. Margenov, editors, Large-Scale Scientific Computing, pages 3–14. Springer International Publishing, 2020.

%\bibitem{Bansal} K. Bansal, P. Rai, and K.K. Sharma. Numerical treatment for the class of time dependent singularly perturbed parabolic problems with general shift arguments. Differ. Equ. Dyn. Syst., 25(2):327–346, 2017.

\bibitem{BFLR} M. Brdar, S. Franz, L. Ludwig, H.-G. Roos. Numerical analysis of a singularly perturbed
convection diffusion problem with shift in space, Appl. Numer. Math. 186 (2023) 129-142.

%\bibitem{Chakravarthy} P.P. Chakravarthy and K. Kumar. An adaptive mesh method for time-dependent singularly perturbed differential-difference equations. Nonlinear Engineering, 8:328– 339, 2019.

%\bibitem{CJLS98} C. Clavero, J. C. Jorge, F. Lisbona, and G. I. Shishkin. A fractional step method on a special mesh for the resolution of multidimensional evolutionary convection-diffusion problems. Appl. Numer. Math., 27(3):211–231, 1998.

\bibitem{Duran} R.G. Dur\'{a}n and A.L. Lombardi. Finite element approximation of convection diffusion
problems using graded meshes, Appl. Numer. Math., 56 (2006) 1314-1325.

%\bibitem{EJT} K. Eriksson, C. Johnson, and V. Thom\'{e}e. Time discretization of parabolic problems by the discontinuous Galerkin method. RAIRO Mod´el. Math. Anal. Num´e., 19:611– 643, 1985.

%\bibitem{Fienberg} S.E. Fienberg. Stochastic models for a single neuron firing trains: a survey. Biometrics 30:399–427, 1974.




%\bibitem{FrM16} S. Franz and G. Matthies. A unified framework for time-dependent singularly perturbed problems with discontinuous Galerkin in time. Math. Comp., 87(313):2113– 2132, 2018.

%\bibitem{FrR19} S. Franz and H.-G. Roos. Error estimates in balanced norms of finite element methods for higher order reaction-diffusion problems. Int. J. Numer. Anal. Model., 17:532– 542, 2020.

\bibitem{Glizer} V.Y. Glizer. Asymptotic solution of a boundary-value problem for linear singularlyperturbed functional differential equations arising in optimal control theory. J. Optim. Theory Appl., 106(2): 309-335, 2000.

%\bibitem{Gupta} V. Gupta, M. Kumar, and S. Kumar. Higher order numerical approximation for time dependent singularly perturbed differential-difference convection-diffusion equations.Numer. Methods Partial Differential Equations, 34:357–380, 2018.

%\bibitem{Jamet} P. Jamet. Galerkin-type approximations which are discontinuous in time for parabolic equations in a variable domain. SIAM J. Appl. Math., 15:912–928, 1978.

%\bibitem{KumarKadalbajoo} D. Kumar and M.K. Kadalbajoo. A parameter-uniform numerical method for timedependent singularly perturbed differential-difference equations. Appl. Math. Model.,35:2805–2819, 2011.

%\bibitem{Devendra} D. Kumar and P. Kumari. Parameter-uniform numerical treatment of singularly perturbed initial-boundary value problems with large delay. Appl. Numer. Math.,153:412–429, 2020.

%\bibitem{LM1} C.G. Lange and R.M. Miura. Singular perturbation analysis of boundary-value problems for differential-difference equations. SIAM J. Appl. Math., 42(3):502–531, 1982.

%\bibitem{LM2} C.G. Lange and R.M. Miura. Singular perturbation analysis of boundary-value problems for differential-difference equations. V. Small shifts with layer behavior. SIAMJ. Appl. Math., 54(1):249–272, 1994.

%\bibitem{LM3} C.G. Lange and R.M. Miura. Singular perturbation analysis of boundary-value problems for differential-difference equations. VI. Small shifts with rapid oscillations.SIAM J. Appl. Math., 54(1):273–283, 1994.

%\bibitem{LS11} R. Lin and M. Stynes. A balanced finite element method for singularly perturbed reaction-diffusion problems. SIAM J. Numerical Analysis, 50(5):2729–2743, 2012.

\bibitem{LongtinMilton} A. Longtin and J.G. Milton. Complex oscillations in the human pupil light reflex with ”mixed” and delayed feedback. Math. Biosci., 90: 183-199, 1988.

%\bibitem{MS21} N. Madden and M. Stynes. A weighted and balanced FEM for singularly perturbed reaction-diffusion problems. Calcolo, 58(2):Paper No. 28, 16, 2021.

%\bibitem{MusilaL} M. Musila and P. L\'{a}nsk\'{y}. Generalized Stein’s model for anatomically complex neurons. BioSystems, 25:179–191, 1991.

%\bibitem{NX13} S. Nicaise and Chr. Xenophontos. Robust approximation of singularly perturbed delay differential equations by the hp finite element method. Comput. Methods Appl.Math., 13(1):21–37, 2013.
\bibitem{RBT} G. Radojev, M. Brdar, Lj. Teofanov, A new approach to a Gartland-type mesh, in preparation, 2023.



\bibitem{Reibiger} C. Reibiger, H.-G. Roos, Numerical analysis of a system of singularly perturbed convection-diffusion equations related to optimal control, Numer. Math., Theory Methods Appl. 4(4) (2011) 562-575, https://doi .org /10 .4208 /nmtma .2011.m1101.

\bibitem{Roos} H.-G. Roos, Layer-adapted meshes for weak boundary layers, https://doi .org /10 .48550 /arXiv.2204 .06188, 2022.
%\bibitem{RST} H. G. Roos, M. Stynes, L. Tobiska, Numerical methods for singularly perturbed differential equations. Convection-diffusion and flow problems,Springer-Verlag, New York, second edition, 2008.
\bibitem{RS} P. Rai, K. K. Sharma, Singularly perturbed convection-diffusion turning point problem
with shifts, in: Mathematical analysis and its applications, Vol. 143 of Springer
Proc. Math. Stat., Springer, New Delhi, 2015, pp. 381-391.

\bibitem{RST} H. G. Roos, M. Stynes, L. Tobiska, Numerical methods for singularly perturbed differential equations. Convection-diffusion and flow problems,Springer-Verlag, New York, second edition, 2008.

%\bibitem{RL99} H.-G. Roos and T. Linß. Sufficient conditions for uniform convergence on layeradapted grids. Computing, 63:27–45, 1999.

%\bibitem{RSch15} H.-G. Roos and M. Schopf. Convergence and stability in balanced norms of finite element methods on Shishkin meshes for reaction-diffusion problems. ZAMM,95(6):551–565, 2015.

\bibitem{SR1} V. Subburayan, N. Ramanujam, Asymptotic initial value technique for singularly
perturbed convection-diffusion delay problems with boundary and weak interior
layers, Appl. Math. Lett. 25 (12) (2012) 2272-2278.

\bibitem{SR2} V. Subburayan, N. Ramanujam, An initial value technique for singularly perturbed
convection-diffusion problems with a negative shift, J. Optim. Theory Appl. 158 (1)
(2013) 234-250.





%\bibitem{Segundo} J.P. Segundo, D.H. Perkel, H. Wyman, H. Hegstad, G.P. Moore. Input-output relations in computer simulated nerve cell: influence of the statistical properties, strength, number and inter-dependence of excitatory pre-dependence of excitatory pre-synaptic terminals. Kybernetik 4:157–171, 1968.

%\bibitem{Stein} R.B. Stein. A theoretical analysis of neuronal variability. Biophys. J., 5:173–194,1965.

%\bibitem{Thomee} V. Thom\'{e}e. Galerkin Finite Element Methods for Parabolic Problems. in Computational Mathematics. Springer-Verlag Berlin Heidelberg, Netherlands, 2006.

\bibitem{Wazewska} M. Wazewska-Czyzewska, A. Lasota, Mathematical models of the red cell system Mat. Stos., 6: 25-40, 1976.

%\bibitem{Zarin14} H. Zarin. On discontinuous Galerkin finite element method for singularly perturbed delay differential equations. Appl. Math. Lett., 38:27–32, 2014.


\end{thebibliography}
\end{document}